\documentclass[a4paper]{amsart}
\usepackage{amsfonts, amscd}
\usepackage{amsmath, enumerate, vmargin}
\usepackage{dsfont}

\newtheorem{theorem}{Theorem}[section]
\newtheorem{lemma}[theorem]{Lemma}

\newtheorem{proposition}[theorem]{Proposition}
\theoremstyle{remark}
\newtheorem{remark}[theorem]{Remark}
\theoremstyle{definition}
\newtheorem{definition}[theorem]{Definition}

\numberwithin{equation}{section}
\makeatother

\newcommand{\A}{{\mathsf A}}
\newcommand{\D}{{\mathsf D}}
\newcommand{\M}{{\mathsf M}}

\newcommand{\T}{{\mathbb T}}
\newcommand{\I}{{\mathds {1}}}

\begin{document}

\title[A noncommutative Helson-Szeg\"o theorem]{A Helson-Szeg\"o theorem for subdiagonal subalgebras with applications to Toeplitz operators}
\author{Louis E Labuschagne}
\address{Internal Box 209, School of Comp., Stat. \& Math. Sci., NWU, Pvt. Bag X6001, 2520 Potchefstroom, South
Africa}
\email{louis.labuschagne@nwu.ac.za}
%\thanks{The contributions of the first author is based upon research supported by the National Research Foundation. Any opinion, findings and conclusions or recommendations expressed in this material, are those of the authors, and therefore the NRF do not accept any liability in regard thereto.}

\author{Quanhua Xu}%$^*$}
\address{School of Mathematics and Statistics, Wuhan University, Wuhan 430072, China and  Laboratoire de Math\'ematiques, Universit\'e de Franche-Comt\'e, 25030 Besancon, cedex-France}
\email{qxu@univ-fcomte.fr}
%\thanks{Xu is partially supported by ANR-2011-BS01-008-01.}


%\date{\today}

\maketitle

\begin{abstract}
We formulate and establish a noncommutative version of the well known Helson-Szeg\"o theorem about the angle between past and future for subdiagonal subalgebras. We then proceed to use this theorem to characterise the symbols of invertible Toeplitz operators on the noncommutative Hardy spaces associated to subdiagonal subalgebras.
\end{abstract}

\bigskip

%\makeatletter
% \renewcommand{\@makefntext}[1]{#1}
%\makeatother \footnotetext{\noindent
%$^*$ Corresponding author}

%%%%%%%%%%%%%%%%%%%%%%%%%%%%%%%%%%%%%%%%%%%%%%%%%%%
%%%%%%%%%%%%%%%%%%%%%%%%%%%%%%%%%%%%%%%%%%%%%%%%%%%

\section{Introduction}

%%%%%%%%%%%%%%%%%%%%%%%%%%%%%%%%%%%%%%%%%%%%%%%%%%%
%%%%%%%%%%%%%%%%%%%%%%%%%%%%%%%%%%%%%%%%%%%%%%%%%%%

Let $\T$ be the unit circle of the complex plane equipped with normalised Lebesgue measure $dm$. We denote by $H^p(\T)$ the usual Hardy spaces on $\T$. Let $P_+$ be the orthogonal projection from $L^2(\T)$ onto $H^2(\T)$. The classical Helson-Szeg\"o theorem \cite{HSz} (see also \cite[section~IV.3]{gar}) characterises those positive measures $\mu$ on $\T$ such that $P_+$ is bounded on $L^2(\T, \mu)$. The condition is that $\mu$ is absolutely continuous with respect to $dm$ and the corresponding Radon-Nikod\'ym derivative $w$ satisfies
 \begin{equation}\label{hs-cond}
 w=e^{u+\widetilde v}\;\textrm{ for two functions }\; u, v\in L^\infty(\T) \;\textrm{ with }\; \|\widetilde v\|_\infty<\pi/2,
 \end{equation}
where $\widetilde v$ denotes the conjugate function of $v$. 

The motivation of this theorem comes from univariate prediction theory. Let $\mathcal P_+$ denote the space of all polynomials in $z$, and $\mathcal P_-$ the space of all polynomials in $\bar z$ without constant term. $\mathcal P=\mathcal P_++\mathcal P_-$ is the space of all trigonometric polynomials. Then $P_+$ is bounded on $L^2(\T,\mu)$ if and only if $\mathcal P_+$ and $\mathcal P_-$ are at positive angle in $L^2(\T,\mu)$. Recall that the angle between  $\mathcal P_+$ and $\mathcal P_-$ is defined as ${\rm arccos}$ of the following quantity
 $$\rho=\sup\big\{\big|\int_{\T}f\bar g d\mu\big|: f\in \mathcal P_+,  g\in \mathcal P_-, \|f\|_{L^2(\T,\mu)}=\|g\|_{L^2(\T,\mu)}=1\big\}.$$
Thus $P_+$ is bounded on $L^2(\T,\mu)$ if and only if $\rho<1$.

In multivariate prediction theory one needs to consider the matrix-valued extension of the Helson-Szeg\"o theorem. Let $\mathbb M_n$ denote the full algebra of complex $n\times n$-matrices, equipped with the normalised trace ${\rm tr}$. Let $\mathcal P_+(\mathbb M_n)$ denote the space of all polynomials in $z$ with coefficients in $\mathbb M_n$.  $\mathcal P_-(\mathbb M_n)$ and  $\mathcal P(\mathbb M_n)$ have similar meanings.  Let $w$ be an $\mathbb M_n$-valued weight on $\T$, i.e. $w$ is an integrable function on $\T$ with values in the family of semidefinite nonnegative matrices. For any trigonometric polynomials $f$ and $g$ in $\mathcal P(\mathbb M_n)$  define
 $$\langle f,\, g\rangle_w=\int_{\T}{\rm tr}(g^*fw)dm\quad\textrm{and}\quad \|f\|_w=\langle f,\, f\rangle_w^{1/2},$$
where $a^*$ denotes the adjoint of a matrix $a$. Like in the scalar case, define
 $$\rho=\sup\big\{\big|\int_{\T}{\rm tr}(g^*fw)dm\big|:  
 f\in\mathcal P_+(\mathbb M_n) , g\in\mathcal P_-(\mathbb M_n), \|f\|_{w}=\|g\|_{w}=1\big\}.$$
Again, $\rho<1$ if and only if $P_+\otimes {\rm Id}_{\mathbb M_n}$ is bounded on $\mathcal P(\mathbb M_n)$  with respect to $\|\,\|_w$. The problem here is, of course, to characterise $w$ such that $\rho<1$ in a way similar to the scalar case.  This time the task is much harder, and it is impossible to find a characterisation as nice as \eqref{hs-cond}. Numerous works have been devoted to this subject, see, for instance \cite{BruDom, bekU, Dom, Pour, P2, TV}. In particular, Pousson's characterisation in \cite{P2} is the matrix-valued analogue of a key intermediate step to \eqref{hs-cond}. It is strong enough for applications to the invertibility of Toeplitz operators.

The preceding two cases can be put into the more general setting of subdiagonal algebras in the sense of \cite{ar}. We will provide an extension of the Helson-Szeg\"o theorem in this general setting. This is the first objective of the paper.

Our second objective is to study the invertibility of Toeplitz operators. It is well known that the Helson-Szeg\"o theorem is closely related to the invertibility of Toeplitz operators. This relationship was remarkably exploited by Devinatz \cite{Dev}. Pousson \cite{P1, P2} then subsequently extended Devinatz's work to the matrix-valued case. Using our extension of the Helson-Szeg\"o theorem, we will characterise the symbols of invertible Toeplitz operators in the very general setting of subdiagonal algebras.

We end this introduction by mentioning the link between the Helson-Szeg\"o theorem and Muckenhoupt's $A_2$ weights. Let $w$ be a weight on $\T$. Hunt,  Muckenhoupt and  Wheeden \cite{HMW} proved that the Riesz projection $P_+$ is bounded on $L^2(\T, w)$ if and only if 
  \begin{equation}\label{A2}
  \sup\frac1{|I|}\int_Iw\, \frac1{|I|}\int_Iw^{-1}<\infty,
  \end{equation}
where the supremum runs over all arcs of $\T$. Such a $w$ is called an $A_2$-weight. Thus for a weight $w$  the two conditions \eqref{hs-cond} and \eqref{A2} are equivalent via the boundedness of the Riesz projection. It seems that it is still an open problem to find a direct proof of this equivalence.

Hunt,  Muckenhoupt and  Wheeden's theorem was extended to the matrix-valued case by Treil and Volberg \cite{TV}. Namely,  let $w$ now be an $\mathbb M_n$-valued weight on $\T$. Then 
 $P_+\otimes {\rm Id}_{\mathbb M_n}$ is bounded on $\mathcal P(\mathbb M_n)$  with respect to $\|\,\|_w$ if and only if 
 $$\sup_I\Big\|\Big(\frac1{|I|}\int_Iw\Big)^{1/2}\,\Big(\frac1{|I|}\int_Iw^{-1}\Big)\,
 \Big(\frac1{|I|}\int_Iw\Big)^{1/2}\Big\|_{\mathbb M_n}<\infty.$$

It is not clear for us how to extend Treil and Volberg's theorem to the case of subdiagonal algebras. On the other hand, Hunt,  Muckenhoupt and  Wheeden also characterised the boundedness of $P_+$ on $L^p(\T,w)$ for any $1<p<\infty$ by the so-called $A_p$ weights. A well known open problem in matrix-valued harmonic analysis is to extend this result to the matrix-valued case; even to the very general one of subdiagonal algebras.

%%%%%%%%%%%%%%%%%%%%%%%%%%%%%%%%%%%%%%%%%%%%%%%%%%%
%%%%%%%%%%%%%%%%%%%%%%%%%%%%%%%%%%%%%%%%%%%%%%%%%%%

\section{Preliminaries}

%%%%%%%%%%%%%%%%%%%%%%%%%%%%%%%%%%%%%%%%%%%%%%%%%%%
%%%%%%%%%%%%%%%%%%%%%%%%%%%%%%%%%%%%%%%%%%%%%%%%%%%

Throughout the paper $\M$ will be a von Neumann algebra possessing a faithful normal
tracial state $\tau$. The associated noncommutative $L^p$-spaces are denoted by $L^p(\M)$.
We refer to \cite{px} for noncommutative integration. For a subset $S$ of
$L^p(\M)$, we will write $[S]_p$ for the closure of $S$ in the $L^p$-topology. On the other hand,
$S^*$ will denote the set of all Hilbert-adjoints of elements of $S$. When an actual Banach dual of some
Banach space is in view, we will for the sake of avoiding confusion prefer
the superscript $\star\,$. For example the dual of $\M$ will be denoted by
$\M^\star$. Because $\M$ is finite, there will for any von Neumann subalgebra
$\mathsf N$ of $\M$, always exist a normal contractive projection $\Psi: \M
\to \mathsf N$ satisfying $\tau \circ \Psi = \tau$. This is the so-called normal faithful  conditional expectation onto $\mathsf N$ with respect to
$\tau$.

A {\em finite subdiagonal algebra} of $\M$ is a weak* closed unital subalgebra $\A$ of
$\M$ satisfying the following conditions
 \begin{itemize}
 \item $\A + \A^*$ is weak* dense in $\M$;
 \item the trace preserving conditional expectation $\Phi : \M \to \A \cap \A^*  = \D$ is multiplicative on $\A$:
 $$
 \Phi(ab) = \Phi(a)  \Phi(b) , \quad a, b \in \A .
 $$
 \end{itemize}
In this case, $\D$ is called the {\em diagonal} of $\A$. We also set $\A_0 = \A \cap {\rm Ker}(\Phi)$.  In the sequel, $\A$ will always denote a finite subdiagonal algebra of $\M$.

Subdiagonal algebras are our noncommutative $H^\infty$'s. The most important example is, of course, the classical $H^\infty(\T)$ on the unit circle. Another example important for multivariate prediction theory is the matrix-valued $H^\infty(\T)$. More precisely, let $\M=L^\infty(\mathbb
T)\otimes\mathbb M_n=L^\infty(\mathbb T;\mathbb M_n)$ equipped with the product trace, and let $\A=H^\infty(\mathbb T;\mathbb M_n)$ -- the subalgebra of $\M$ consisting of $n\times n$-matrices with entries in $H^\infty(\T)$. Many classical results about Hardy spaces on $\T$ have been transferred to the matrix-valued case. A third example is the upper triangle subalgebra $\T_n$ of $\mathbb M_n$. This example is closely related to the second one, and is a finite dimensional nest algebra. We refer to \cite[\S 8]{px} for more information and historical references on subdiagonal algebras, in particular, on matrix-valued analytic functions.

For $p<\infty$ the Hardy space $H^p(\A)$ associated with a finite subdiagonal algebra $\A$ is defined to be $[\A]_p$. The closure of $\A_0$ in $L^p(\M)$ will be denoted by $H^p_0(\M)$. By convention, we put $H^\infty(\A)=\A$ and $H^\infty_0(\A)=\A_0$. These spaces exhibit many of the properties of classical $H^p$ spaces (see \cite{bx, BL3, BLsur, MW1, ran, S}). In particular for $1<p<\infty$, $L^p(\M)$ appears as the Banach space direct sum of $H^p(\M)$ and $H^p_0(\M)^*$, with $H^p(\M)$ appearing as the Banach space direct sum of $H^p_0(\M)$ and $L^p(\D)$. In the case $p=2$, these direct sums are even orthogonal direct sums.

Recall that if a weight $w$ on $\T$ satisfies \eqref{hs-cond}, then necessarily $\log w\in L^1(\T)$, or equivalently,
  \begin{equation}\label{g-mean}
  \exp\big(\int_\T \log w\big)>0.
  \end{equation}\label{hs}
The integrability of $\log w$ is also equivalent to the existence of an outer function $h\in H^1(\T)$ such that $w=|h|$. To state the outer-inner factorisation and prove the Helson-Szeg\"o analogue for subdiagonal algebras, we need an appropriate substitute of the latter condition. This is achieved by the \emph{Fuglede-Kadison determinant}. Recall that the Fuglede-Kadison determinant $\Delta(a)$
of an operator $a\in L^p(\M)$ ($p>0$) can be defined by
 $$\Delta(a)=\exp\big(\tau(\log|a|)\big)
 =\exp\big(\int_0^\infty\log t\,d\nu_{|a|}(t)\big),$$
where $d\nu_{|a|}$ denotes the probability measure on $\mathbb R_+$
which is obtained by composing the spectral measure of $|a|$ with
the trace $\tau$. It is easy to check that
 $$\Delta(a)=\lim_{p\to0}\|a\|_p\quad\textrm{and}\quad 
 \Delta(a) = \inf_{\epsilon > 0}\exp\tau(\log(|a|+\epsilon \I))\,.$$
As the usual determinant of matrices, $\Delta$ is also
multiplicative: $\Delta(ab)=\Delta(a)\Delta(b)$. We refer the reader for
information on determinant to \cite{fug-kad, ar} in the
case of bounded operators, and to \cite{Br, HS} for
unbounded operators.

Return to our Hardy spaces. An element $h$ of $H^p(\M)$ with $p<\infty$ is said to be an {\em outer} element if $h\A$ is dense in $H^p(\M)$. If in addition $\Delta(h)>0$, we call such an $h$ {\em strongly outer}. For an analysis of outer elements in the present context, we refer the interested reader to \cite{BL3} for $p\ge 1$ and \cite{bx} for $p<1$. We will however pause to summarise the essential points of the theory. For any outer element $h$ of $H^p(\M)$, both $h$ and $\Phi(h)$ necessarily have dense range and trivial kernel. Hence their inverses exist as affiliated operators. For such an outer element, we also necessarily have that $\Delta(h)=\Delta(\Phi(h))$. If indeed $\Delta(h)>0$, the equality $\Delta(h)=\Delta(\Phi(h))$ is sufficient for $h$ to be outer. Using this fact it is now an easy exercise to see that if $\Delta(h)>0$, then $h$ is an outer element of $H^p(\M)$ if and only if $h^*$ is an outer element of $H^p(\M)^*$ if and only if $h$ is right outer in the 
 sense that $\A h$ will also be dense in $H^p(\M)$. In this theory one also has a type of noncommutative Riesz-Szeg\"o theorem, in that any $f\in L^p(\M)$ for which $\Delta(f)>0$, may be written in the form $f=uh$ where $u\in\M$ is unitary and $h\in H^p(\M)$ an outer element of $H^p(\M)$.

Given a state $\omega$ on $\M$, we write $(\pi_\omega, L^2(\omega),  \Omega_\omega)$ for the cyclic representation associated to $\omega$. The subspaces $\A^*$ and $\A_0$ embed canonically into $L^2(\omega)$ by means of the operation $a \mapsto \pi_\omega(a)\Omega_\omega$. The angle between $\A^*$ and $\A_0$ in $L^2(\omega)$ is defined to be that between the closed subspaces  $\overline{\pi_\omega(\A^*)\Omega_\omega}$ and $\overline{\pi_\omega(\A_0)\Omega_\omega}$. The latter is equal to $\arccos \rho$ with $\rho$ given by
 $$\rho = \sup\{|\langle\pi_\omega(a)\Omega_\omega, \pi_\omega(b)\Omega_\omega\rangle| : 
 a \in \A_0, b \in \A^*, \|\pi_\omega(a)\Omega_\omega\| \leq 1, \|\pi_\omega(b)\Omega_\omega\| \leq 1\}.$$
In view of the fact that $\langle\pi_\omega(a)\Omega_\omega, \pi_\omega(b)\Omega_\omega\rangle = \omega(b^*a)$, this may be rewritten as
 $$\rho = \sup\{|\omega(b^*a)| : a \in\A_0, b \in \A^*, \omega(|a|^2) \leq 1, \omega(|b|^2) \leq 1\}.$$
In general $0 \leq \rho \leq 1$. $\A^*$ and $\A_0$ are said to be at positive angle in  $L^2(\omega)$ if $\rho<1$. Let $P_+$ be the orthogonal projection from $L^2(\M)$ onto $H^2(\M)$. It is then clear that $P_+$ defines a bounded operator on  $L^2(\omega)$ if and only if $\rho<1$. 

%%%%%%%%%%%%%%%%%%%%%%%%%%%%%%%%%%%%%%%%%%%%%%%%%%%
%%%%%%%%%%%%%%%%%%%%%%%%%%%%%%%%%%%%%%%%%%%%%%%%%%%

\section{A noncommutative Helson-Szeg\"o theorem}

%%%%%%%%%%%%%%%%%%%%%%%%%%%%%%%%%%%%%%%%%%%%%%%%%%%
%%%%%%%%%%%%%%%%%%%%%%%%%%%%%%%%%%%%%%%%%%%%%%%%%%%

In this section we present our noncommutative Helson-Szeg\"o theorem. This theorem will prove to be an important ingredient in our onslaught on Toeplitz operators in the next section. As recalled previously, the classical Helson-Szeg\"o theorem  contains the information that any finite Borel measure for which the angle between $\A$ and $\A^*_0$ is positive must necessarily be absolutely continuous with respect to Lebesgue measure, and moreover that the Radon-Nikod\'ym derivative of this measure must have a strictly positive geometric mean \eqref{g-mean}. Before presenting our noncommutative Helson-Szeg\"o theorem, we first show that under some mild restrictions the same claims are true in the noncommutative case. $L^p_+(\M)$ will denote the positive cone of $L^p(\M)$.

\begin{proposition}\label{prop1}
Let $\D = \A \cap \A^*$ be finite dimensional, and let $\omega$ be a state on $\M$ for which $\rho < 1$. Then $\omega$ is of the form $\omega = \tau(g \cdot)$ for some $g \in L^1_+(\M)$.
\end{proposition}

\begin{proof}
 We keep the notation introduced at the end of the previous section. Let $\omega_n$ and $\omega_s$ respectively be the normal and singular parts of  $\omega$. Firstly note that by \cite[III.2.14]{Tak1}, there exists a central projection $e_0$ in
$\pi_{\omega}(\M)''$ such that for any $\xi, \psi \in
L^2(\omega)$ the functionals $a \mapsto
\langle\pi_{\omega}(a)e_0\xi, \,\psi\rangle$ and $a \mapsto
\langle\pi_{\omega}(a)e_0^\perp\xi,\, \psi\rangle$ on $\M$ are
respectively the normal and singular parts of the functional $a
\mapsto \langle\pi_{\omega}(a)\xi, \,\psi\rangle$, where $e_0^\perp= \I-e_0$.  In particular,
the triples $(e_0\pi_{\omega},
e_0L^2(\omega), e_0\Omega_{\omega})$ and
$(e_0^\perp\pi_{\omega}, e_0^\perp L^2(\omega),
e_0^\perp\Omega_{\omega})$ are copies of the GNS representations
of $\omega_n$ and $\omega_s$ respectively.

Since $\rho<1$,  we must have that 
 $$\overline{\pi_\omega(\A_0)\Omega_\omega} \cap \overline{\pi_\omega(\A^*)\Omega_\omega} = \{0\}.$$
Now suppose that the singular part $\omega_s$ of $\omega$ is nonzero. By Ueda's noncommutative peak-set theorem \cite[Theorem~1]{Ueda} there exist an orthogonal projection $e$ in the second dual $\M^{\star\star}$ of $\M$ and a contractive element $a$ of $\A$ so that
\begin{itemize}
\item $a^n$ converges to $e$ in the weak*-topology on $\M^{\star\star}$;
\item $\omega_s(e) = \omega_s(\I)$ (here $\omega_s$ is identified with its canonical extension to $\M^{\star\star}$);
\item $a^n$ converges to 0 in the weak*-topology on $\M$.
\end{itemize}
Since the expectation $\Phi$ is weak*-continuous on $\M$, $\Phi(a^n)$ is weak* convergent to 0. But then the finite dimensionality of $\D$ ensures that $\Phi(a^n)$ converges to 0 in norm.

Recall that the bidual $\M^{\star\star}$ of $\M$ may be represented as the double commutant of $\M$ in its universal representation. So when this realisation of $\M^{\star\star}$ is compressed to the specific representation engendered by $\omega$, it follows that $e$ yields a projection $\widetilde e$ in $\pi_\omega(\M)''$ to which $\pi_\omega(a^n)$ converges in the weak*-topology on $\pi_\omega(\M)''$. This weak* convergence in $\pi_\omega(\M)''$ together with the second bullet above, then yield the facts that
\begin{itemize}
\item $\pi_\omega(a^n)\Omega_\omega$ converges to $\widetilde e\,\Omega_\omega$ in the weak-topology on $L^2(\omega)$;
\item $\langle \widetilde e\,\Omega_\omega,\, e_0^\perp\Omega_\omega\rangle = \omega_s(\I)$.
\end{itemize}
From the first bullet and the fact that
$\{\Phi(a^n)\}$ is a norm-null sequence, it follows that $\pi_\omega(a^n - \Phi(a^n))\Omega_\omega$ is weakly convergent to $\tilde e\,\Omega_\omega$, and hence that $\widetilde e\,\Omega_\omega \in \overline{\pi_\omega(\A_0)\Omega_\omega}$. But if $a^n$ converges to $e$ in the weak*-topology on $\M^{\star\star}$, then surely so does $(a^*)^n$. In terms of the GNS representation for $\omega$, this means that $\pi_\omega((a^*)^n)\Omega_\omega$ also converges to $\widetilde e\,\Omega_\omega$ in the weak-topology on $L^2(\omega)$. But then $\widetilde e\,\Omega_\omega \in \overline{\pi_\omega(\A^*)\Omega_\omega}$. Then $\widetilde e\,\Omega_\omega = 0$ since $\widetilde e \,\Omega_\omega\in \overline{\pi_\omega(\A_0)\Omega_\omega} \cap \overline{\pi_\omega(\A^*)\Omega_\omega}.$ But this cannot be, since by the second bullet this would mean that $\omega_s(\I) = \langle \widetilde e\,\Omega_\omega, \, e_0^\perp\Omega_\omega\rangle = 0$. Thus our supposition that $\omega_s$ is nonz
 ero, must be false. The condition that $\rho < 1$, is therefore sufficient to force $\omega$ to be normal. That is $\omega$ is of the form $\omega = \tau(g \cdot)$ for some $g \in L^1_+(\M)$.
\end{proof}

The following lemmata  present two known elementary facts. 

\begin{lemma}\label{Phisupp}
 For any $g \in L^1_+(\M)$ we have that 
 $$s(\Phi(g)) \geq s(g),$$
where $s(g)$ denotes the support projection of $g$.
 \end{lemma}

\begin{proof} For simplicity of notation we respectively write $s$ and $s_\Phi$ for $s(g)$ and $s(\Phi(g))$. Since
$s_\Phi \in \D$, we have that 
 $$\tau(s_\Phi^\perp gs_\Phi^\perp) = \tau\circ\Phi(s_\Phi^\perp gs_\Phi^\perp) =
 \tau(s_\Phi^\perp\Phi(g)s_\Phi^\perp) = 0.$$
Therefore $g^{1/2}s_\Phi^\perp = s_\Phi^\perp g^{1/2} = 0$. This is sufficient
to force $s_\Phi^\perp \perp s$, which in turn suffices to show that $s_\Phi \geq s$.
\end{proof}

\begin{lemma}
Let $e$ be a nonzero projection in $\D$. Then $e\A e$ is a finite maximal subdiagonal subalgebra of $e\M e$ $($equipped with the trace $\tau_e(\cdot) = \frac{1}{\tau(e)}\tau(\cdot))$ with diagonal $e\A e \cap (e\A e)^* = e\D e$.
\end{lemma}

\begin{proof}
The expectation $\Phi$ is trivially still multiplicative on the compression $eAe$. Using the fact that $e \in \D$, it is an exercise to see that $\Phi$ maps $e\A e$ onto $e\D e$. It is also straightforward to see that the weak*-density of $\A+\A^*$ in $\M$ forces the weak*-density of $e\A e +(e\A e)^*$ in $e\M e$, and that $(e\A e)_0=e\A_0 e$.
\end{proof}

\begin{definition}
Adopting the notation of the previous two lemmata, given a nonzero element $g \in L^1_+(\M)$, we define
$\Delta_\Phi(g)$ to be the determinant of $s_\Phi g s_\Phi$ regarded as an element of $(s_\Phi \M s_\Phi, \tau_{s_\Phi})$
\end{definition}

\begin{proposition}\label{prop2}
Let $\D = \A \cap \A^*$ be finite dimensional, and let $g \in L^1_+(\M)$ be a norm-one element for which the state $\omega = \tau(g\cdot)$ satisfies $\rho < 1$. Then $\Delta_\Phi(g) > 0$.
\end{proposition}

\begin{proof}
It is clear from the previous lemmata that we may reduce matters to the case where $s(\Phi(g)) = \I$, and hence we will assume this to be the case. Suppose by way of contradiction that $\Delta(g) = 0$. By the Szeg\"o formula for subdiagonal algebras \cite{L1}, we then have that
 $$0 = \Delta(g) = \inf \{ \tau(g |a-d|^2): a \in \A_0 , d \in \D , \Delta(d) \geq 1 \}.$$
Thus there exist sequences $\{a_n\} \subset \A_0$ and $\{d_n\} \subset \D$ with $\Delta(d_n) \geq 1$ for all $n$, so that 
 $$\tau(g|a_n-d_n|^2) \to 0 \quad \mbox{as} \quad n \to \infty.$$
By Lemma 2.2 of \cite{BL2} we may assume all the $d_n$'s to be invertible. Now let $u_n \in \D$ be the unitary in the polar decomposition $d_n = u_n|d_n|$. It is an exercise to see that then $\{u_n^*a_n\} \subset \A_0$ with $|a_n-d_n|^2 = \big|u_n^*a_n-|d_n|\big|^2$. Making the required replacements, we may therefore also assume that $\{d_n\} \subset \D^+$.

Since $1 \leq \Delta(d_n) \leq \|d_n\|_\infty$ for all $n$, we will for the sequences $\widetilde{d_n} = \frac{1}{\|d_n\|_\infty}d_n$ and $\widetilde{a_n} = \frac{1}{\|d_n\|_\infty}a_n$ ($n \in \mathbb{N}$), still have that $\tau(g|\widetilde{a_n}-\widetilde{d_n}|^2) \to 0$ as $n \to \infty$. Now recall that $\D$ is finite dimensional. So by passing to a subsequence if necessary, we may assume that $\{\widetilde{d_n}\}$ converges uniformly to some norm one element $d_0$ of $\D^+$. But then by what we showed above,
 \begin{eqnarray*}
 \|\pi_g(\widetilde{a_n}) - \pi({d_0})\|_2 &=& \tau(g|\widetilde{a_n}-{d_0}|^2)^{1/2}\\
 &\leq& \tau(g|\widetilde{a_n}-\widetilde{d_n}|^2)^{1/2} + \tau(g|\widetilde{d_n}-{d_0}|^2)^{1/2}\\
 &\leq& \tau(g|\widetilde{a_n}-\widetilde{d_n}|^2)^{1/2} + \|\widetilde{d_n}-{d_0}\|_\infty\tau(g)^{1/2}\\
 &\to& 0.
 \end{eqnarray*}
Thus $\pi_g(d_0) \in \overline{\pi_g(\A_0)} \cap \overline{\pi_g(\A^*)}$. Since $\Phi(g)$ is of full support, we have that
$\Phi(g)^{1/2}d_0\Phi(g)^{1/2} \neq 0$. So 
 $$0 < \tau(\Phi(g)^{1/2}d_0\Phi(g)^{1/2}) = \tau(\Phi(g)d_0) = \tau(\Phi(gd_0)) =\tau(gd_0).$$
Therefore $\pi_g(d_0) \neq 0$. But this proves that the subspaces $\overline{\pi_g(\A_0)}$ and $\overline{\pi_g(\A^*)}$
have a nonzero intersection, and hence that $\rho = 1$.
\end{proof}

\begin{remark}
 Under the assumption of the previous proposition, the support $s(\Phi(g))$ can be 
strictly less than $\I$. Indeed, consider the $\mathbb M_2$-valued case: $\M=L^\infty(\T;\mathbb M_2)$ and 
$\A=H^\infty(\T;\mathbb M_2)$. Let $w$ be a weight satisfying \eqref{hs-cond} and $g=w\otimes e_{11}$, where $e_{11}$ the matrix whose only nonzero entry is the one at the position $(1, 1)$ which is equal to $1$. Then the corresponding $\rho$ is less than $1$ but $s(\Phi(g))=e_{11}$.
\end{remark}

The following technical lemma is a crucial step in the proof of the classical Helson-Szeg\"o theorem. The challenge one faces in the noncommutative world is that the functional calculus at our disposal in that context is simply not strong enough to reproduce so detailed a statement in that framework. However in the lemma following this one, we present what we believe to be a reasonable noncommutative substitute of this interesting lemma.

\begin{lemma}
Let $u = e^{-i\psi}$ with $\psi$  a real measurable function on $\mathbb{T}$. Then $\inf_{g\in H^\infty(\T)}\|e^{-i\psi}-g\|_\infty < 1$ if and only if there exist an $\epsilon> 0$ and a $k_0\in H^\infty(\T)$ so that $|k_0|\geq\epsilon$ and $|\psi + \arg(k_0)| \leq \frac{\pi}{2}-\epsilon$ almost everywhere .
\end{lemma}

\begin{lemma}
Let $u$ be a unitary element of $\M$. Then there exists some $f\in \A$ so that $\|u-f\|_\infty<1$ if and only if there exists $h\in \A$ so that $\Re(u^*h)$ is strictly positive.
\end{lemma}

\begin{proof}
Suppose first that there exists $f\in \A$ with $\|u-f\|_\infty<1$. We then equivalently have that $\|\I-u^*f\|=\|\I-f^*u\|<1$. On setting $\alpha = \|\I-u^*f\|$, it follows that $\|\I-\Re(u^*f)\| \leq \alpha < 1$, and hence that 
 $$-\alpha\I\leq\Re(u^*f)-\I\leq \alpha\I.$$
This in turn ensures that $0<(1-\alpha)\I \leq \Re(u^*f)$.

Conversely suppose that there exists $h\in \A\cap \M^{-1}$ so that $\Re(u^*h)\geq \alpha\I$ for some $0 < \alpha \leq \|\Re(u^*h)\| \leq\|h\|$, where $\M^{-1}$ denotes the subset of invertible elements of $\M$. Given $\epsilon> 0$, set $\lambda=\frac{\epsilon}{\|h\|}$. It then follows that
 $$-2\lambda\Re(u^*h) + \lambda^2|h|^2 \leq -\left(\tfrac{2\alpha\epsilon}{\|h\|}-\epsilon^2\right)\I.$$
(Observe that $\frac{\alpha}{\|h\|} \leq 1$ in the above inequality.) It is clear that if $\epsilon$ is small enough, we would have that $1>\left(\tfrac{2\alpha\epsilon}{\|h\|}-\epsilon^2\right)>0$. Thus we may assume this to be the case. For simplicity of notation we now set $\delta = \left(\tfrac{2\alpha\epsilon}{\|h\|}-\epsilon^2\right)$. It therefore follows from the previous centered inequality that 
 $$0 \leq |\I-u^*(\lambda h)|^2 = \I -2\Re(u^*(\lambda h)+|\lambda h|^2 \leq (1-\delta) \I.$$
Hence as required, $\|\I-u^*(\lambda h)\|^2\leq(1-\delta)< 1$.
\end{proof}

We are now finally ready to present our noncommutative Helson-Szeg\"o theorem. In view of Propositions \ref{prop1} and \ref{prop2}, it is not unreasonable to restrict attention to normal states $\tau(g\cdot)$ in this theorem for which $\Delta_\Phi(g) > 0$. The following result is a sharpening of the result of Pousson \cite[Theorem 4.3]{P2}, in that here the conditions imposed on the unitary $u$ are less restrictive. This sharpening is achieved by means of the preceding Lemma.

\begin{theorem}\label{HS1}
 Let $g \in L^1_+(\M)$ be given with $\|g\|_1 = 1$, and denote $s(\Phi(g))$ by $s_\Phi$. Consider the state $\omega = \tau(g\cdot)$. Then $\rho < 1$ and $\Delta_\Phi(g) > 0$ if and only if $g$ is of the form $g=f_Ruf_L$ where
 \begin{itemize}
 \item $u\in \M$ is a partial isometry with initial and final projections $s_\Phi$ for which there exists some $k\in s_\Phi \A s_\Phi$ so that $\Re(u^*k) \geq  \alpha s_\Phi$ for some $\alpha >0$,
 \item and $f_L$ and $f_R$ are strongly outer elements of $H^2(\M)$ commuting with $s_\Phi$ for which $g+(\I-s_\Phi) = |f_L|^2 = |f_R^*|^2$.
 \end{itemize}
If in addition $\dim\D<\infty$, we may dispense with the restrictions that $\omega$ is normal, and that $\Delta_\Phi(g) > 0$.
 \end{theorem}

\begin{proof}
 Set $s=s_\Phi$ for simplicity. Suppose that $g$ satisfies the condition $\Delta_\Phi(g) > 0$. Using the fact that then $\Delta_\Phi(g^{1/2}) = \Delta_\Phi(g)^{1/2} > 0$, it follows from the noncommutative Riesz-Szeg\"{o} theorem (see \cite{BL3}) that there exist strongly outer elements $h_L, h_R \in H^2(s \M s)$ and unitaries $v_L, v_R \in s\M s$ for which $g^{1/2} = v_Lh_L = h_Rv_R$. (Then also $g^{1/2} = |h_L| = |h_R^*|$.)
We set 
 $$u = v_Rv_L, \quad f_L = h_L+s^\perp,\quad f_R=h_R+s^\perp.$$
It is then clear that 
 $$g = f_Ruf_L\quad \textrm{and}\quad  g +s^\perp = |f_L|^2 = |f_R^*|^2.$$ 
We proceed to show that $f_L$ and $f_R$ are strongly outer. The proofs of the two cases are identical, and hence we do this for $f_L$ only. Notice that   
 $$\log(|f_L|) = \log(|h_L| + s^\perp) = \log(|h_L|)s.$$ 
Since $\Phi(f_L) = \Phi(h_L)+s^\perp$, we similarly have that 
 $$\log(|\Phi(f_L)|) = \log(|\Phi(h_L)|)s.$$
It then follows that
 $$ \tau(\log|f_L|) =\tau(s) \tau_{s}(\log|h_L|) \quad \textrm{and}\quad 
  \tau(\log|\Phi(f_L)|) =\tau(s) \tau_{s}(\log|\Phi(h_L)|).$$
Thus the outerness of $h_L$ yields that
 $$\tau(\log|f_L|) = \tau(\log|\Phi(f_L)|)>-\infty,\;\textrm{ so }\; \Delta(f_L)=\Delta(\Phi(f_L))>0.$$
Then an application of \cite[Theorem 4.4]{BL3} now shows that $f_L$ is strongly outer.

On the other hand, we have 
 $$\langle \pi_g(a)\Omega_g,\, \pi(b)\Omega_g \rangle = \tau(gb^*a) = \tau(uf_Lb^*af_R),\quad a \in \A_0, b\in \A^*.$$ 
So
 \begin{eqnarray*}
 \rho &=& \sup\{|\tau(gb^*a)| : a \in \A_0, b\in \A^*, \tau(g|a|^2) \leq 1, \tau(g|b|^2) \leq 1\}\\
 &=& \sup\{|\tau\big((u (s f_Lb^*)(af_Rs )\big)| : a \in \A_0, b\in \A^*, 
 \tau(|af_Rs|^2) \leq 1, \tau(|bf_L^*s|^2) \leq 1\}\\
 &=& \sup\{|\tau(uF_1F_2)| : F_1 \in s H^2(\M), F_2 \in H_0^2(\M)s, \|F_1\|_2 \leq 1,  \|F_2\|_2 \leq 1\}.
 \end{eqnarray*}
In the above computation one has used the fact that $f_L$ and $f_R$ are strongly outer to approximate $F_1$ and $F_2$ with elements of the form $sf_Lb^*$ and $af_Rs$ where $a \in \A_0$ and $b\in \A^*$. However, it is easy to check that for $ F_1 \in s H^2(\M), F_2 \in H_0^2(\M)s$
 $$F_1F_2\in H^1_0(s\M s)\;\textrm{ and }\;\|F_1F_2\|_1\le \|F_1\|_2\|F_2\|_2.$$
Conversely, by  the Noncommutative Riesz Factorisation theorem \cite{MW1, S}, for any $\epsilon>0$ and any $F\in H^1_0(s\M s)$ there exist $F_1\in H^2(s\M s)\subset s H^2(\M)$ and $F_2\in H^2_0(s\M s)\subset H^2(\M)s$ such that 
 $$F=F_1F_2 \;\textrm{ and }\; \|F_1\|_2\|F_2\|_2\le \|F\|_1+\epsilon.$$
From these discussions we conclude that
  \begin{eqnarray*}
  \rho &=&\sup\{|\tau(uF)| : F \in H_0^1(s \M s), \|F\|_1\leq 1\}\\
  &=& \sup\{|\tau_{s}(uF)| : F \in H_0^1(s\M s), \tau_{s}(|F|) \leq 1\}.
 \end{eqnarray*}
The norm of the restriction of the functional $L^1(s\M s)\to\mathbb{C}:a\mapsto\tau_{s}(ua)$ to $H_0^1(s\M s)$ is by duality
precisely the norm of the equivalence class $[u]$ in the quotient space $s\M s/(H_0^1(s\M s))^\circ$. However, it is well known that 
 $$s\A s=\{a\in s\M s: \tau_{s}(ab)=0, b\in s\A_0s\}$$
(cf. e.g.,  \cite{S} ). From this fact it is now an easy exercise to see that the polar
$(H_0^1(s\M s))^\circ$ is nothing but $s\A s$. It therefore follows that
 $$\rho = \inf\{\|u-k\|_\infty : k \in s\A s\}.$$ 
The result now follows from an application of the preceding Lemma.
\end{proof}

%%%%%%%%%%%%%%%%%%%%%%%%%%%%%%%%%%%%%%%%%%%%%%%%%%%
%%%%%%%%%%%%%%%%%%%%%%%%%%%%%%%%%%%%%%%%%%%%%%%%%%%

\section{Invertibility of Toeplitz operators}

%%%%%%%%%%%%%%%%%%%%%%%%%%%%%%%%%%%%%%%%%%%%%%%%%%%
%%%%%%%%%%%%%%%%%%%%%%%%%%%%%%%%%%%%%%%%%%%%%%%%%%%

We start by recalling the definition of Toeplitz operators. Given $a\in \M$, the Toeplitz operator $T_a$  with symbol $a$ is defined  to be the map
 $$T_a:H^2(\M)\to H^2(\M):b \mapsto P_+(ab),$$
where $P_+$  denotes the orthogonal projection from $L^2(\M)$ onto $H^2(\M)$. Our basic reference for Toeplitz operators in this context is \cite{MW} (see also \cite{Pr}).  

We will characterise the symbols of invertible Toeplitz operators. We point out that these results are new even for the matrix-valued case.  In achieving this characterisation, we will follow the same basic strategy as Devinatz \cite{Dev} in his remarkable solution of this problem in the classic setting. Our first result essentially reduces the problem to that of characterising invertible Toeplitz operators with unitary symbols.

\begin{theorem}\label{Toep1}
 Let $a\in \M$ be given. A necessary and sufficient condition for $T_a$ to be invertible is that it can be written in the form $a=uk$ where $k \in {\A}^{-1}$, and $u\in \M$ is a unitary for which $T_u$ is invertible.
\end{theorem}

Suppose that $a\in\M$ is indeed of the form $a=uk$ where $k \in {\A}^{-1}$, and $u\in \M$ is a unitary. It is a simple exercise to see that then $T_k$ is invertible with inverse $T_{k^{-1}}$. Since $T_aT_{k^{-1}}=T_u$ and $T_uT_k=T_a$, it is now clear that $T_a$ will then be invertible if and only if $T_u$ is invertible.

\begin{proof}
The sufficiency of the stated condition was noted in the above discussion. To see the necessity, assume $T_a$ to be invertible. There must therefore exist some $g\in H^2(\M)$ so that $T_ag=\I$. This in turn can only be true if there exists some $h\in H^2_0(\M)$ so that $ag=\I+h^*$. By the generalised Jensen inequality \cite[3.3]{BL3} we have that 
 $$\Delta(a)\Delta(g) = \Delta(ag) =\Delta(\I+h^*) \geq \Delta(\Phi(\I+h^*)) = \Delta(\I)=1.$$
Clearly we then have that $\Delta(|a|^{1/2}) = \Delta(a)^{1/2} > 0$. So by the noncommutative Riesz-Szeg\"o theorem \cite[4.14]{BL3}, there must exist an outer element $f\in H^2(\M)$ and a unitary $v$ so that $|a|^{1/2}=vf$. (Note then that $f\in\M$, so $f$ must belong to $\A$ too.) Let $w$ be the unitary in the polar decomposition $a=w|a|$, and consider $b=w|a|^{1/2}v$. Notice that by construction $bf=a$. Thus $T_bT_f=T_a$. We will use this formula to show that $T_f$ is invertible, from which the result will then follow.

Firstly note that the injectivity of $T_a$ combined with the above equality, ensures that $T_f$ is injective. Next notice that the equality $T_bT_f=T_a$ ensures that $(T_a)^{-1}T_b$ is a left inverse for $T_f$. So $T_f$ must have a closed range. However since $f$ is outer, we also have that $[f\A]_2=H_2(\M)$. Since $f\A \subset T_f(H_2(\M))$, these two facts ensure that the range of $T_f$ is all of $H_2(\M)$. Hence $T_f$ must be invertible. 

But if $T_f$ is invertible, then so is $T_f^*=T_{f^*}$. Since  $T_{f^*}T_f = T_{|f|^2} =T_{|a|}$, the operator $T_{|a|}$ must be invertible. Since $\sigma(|a|) \subset \sigma(T_{|a|})$ by Theorem 3.5 of \cite{MW}, we must have that $0\not\in \sigma(|a|)$. In other words $|a|$ must be strictly positive. But if $|a|$ is strictly positive, then by Arveson's factorisation theorem there exists some $k\in \A^{-1}$ with $|a|=|k|$. Finally let $w_0$ be the unitary in the polar form $k = w_0|k|$. Then  $a=ww_0^*k$, which proves the theorem with $u=ww_0^*$.
\end{proof}

Our next step in achieving the desired characterisation, is to present some necessary structural information regarding unitaries $u$ for which $T_u$ is invertible. We then subsequently use this structural information to obtain a characterisation of invertibility in terms of positive angle.

\begin{lemma}\label{invtoep}
 Let $u\in \M$ be a unitary. A necessary condition for $T_u$ to be invertible is that it is of the form $u=(g_1^*)^{-1}dg_0^{-1}$  where $g_0,g_1$ are strongly outer elements of $H^2(\M)$ and $d$ a strongly outer element of $L^2(\D)$ related by the conditions that 
 $$d=\Phi(g_0)=\Phi(g_1^*),\quad dg^{-1}_0, d^*g^{-1}_1\in H^2(\M)\quad\mbox{and}\quad g_0^*g_0=d^*(g_1^*g_1)^{-1}d.$$
\end{lemma}

\begin{proof}
Let $u\in \M$ be a unitary for which $T_u$ is invertible. Since $T_u^*=T_{u^*}$ is then also invertible, it follows that there must exist $g_0, g_1 \in H^2(\M)$ so that $T_ug_0=\I=T_{u^*}g_1$. This in turn means that there exist $h_0,h_1 \in H^2_0(\M)$ with
 $$ug_0=\I+h^*_0,\qquad u^*g_1=\I+h^*_1.$$ 
Notice that we may then apply the generalised Jensen inequality \cite[3.3]{BL3} to conclude that 
 $$\Delta(g_0)=\Delta(u)\Delta(g_0)=\Delta(ug_0)\geq \Delta(\I)=1.$$ 
Similarly $\Delta(g_1)\geq 1$. By \cite[4.2 \& 4.15]{BL3} this means that both $g_0$ and $g_1$ are injective with dense range, and hence that $g_0^{-1}$ and $g_1^{-1}$ exist as affiliated operators. On the other hand, we have that
  $$g^*_1ug_0=g^*_1(\I+h^*_0)\in H^1(\M)^*\quad\textrm{and}\quad g_0^*u^*g_1=g_0^*(\I+h^*_1)\in H^1(\M)^*.$$ 
 Hence 
 $$g^*_1ug_0\in H^1(\M)\cap H^1(\M)^* = L^1(\D).$$
If we denote this element by $d$, it follows that $u$ is of the form $u=(g_1^*)^{-1}d\,g_0^{-1}$. It is then clear that $d^*(g_1^*g_1)^{-1}d=g_0^*g_0$.

It remains to show that $g_0$ and $g_1$ are outer and that $d=\Phi(g_0)=\Phi(g_1^*)$. To see this notice that since $g_1^*\in H^2(\M)^*$ and $ug_0=\I+h^*_0\in H^2(\M)^*$, we have that 
 $$d =\Phi(d)=\Phi(g^*_1ug_0) = \Phi(g^*_1(\I+h^*_0)) =\Phi(g^*_1)\Phi(\I+h^*_0)=\Phi(g_1^*).$$ 
Similarly, $d=\Phi(g_0)$. (Since $\Phi$ maps $H^2(M)$ onto $L^2(\mathcal{D})$, this equality also shows that $d$ is in fact in $L^2(\mathcal{D})$, and not just $L^1(\mathcal{D})$.)  It now follows from the equality $g_0^*g_0=d^*(g_1^*g_1)^{-1}d$, that 
 $$\Delta(g_0)^2=\Delta(g_0^*g_0)=\Delta(d^*(g_1^*g_1)^{-1}d)=\Delta(d^*)^2\Delta(g_1)^{-2}=\Delta(\Phi(g_1))^2\Delta(g_1)^{-2}.$$
Since as was shown earlier we have that $\Delta(g_0)\geq 1$, it therefore follows that $0< \Delta(g_1)\leq \Delta(\Phi(g_1))$. If we combine this with the generalised Jensen inequality \cite[3.3]{BL3}, we obtain $0< \Delta(g_1)=\Delta(\Phi(g_1))$. Similarly, $0< \Delta(g_0)=\Delta(\Phi(g_0))$. Thus by \cite[Theorem 4.4]{BL3}, both $g_0$ and $g_1$ are strongly outer.
\end{proof}

When combined with Theorem \ref{Toep1}, the following lemma characterises the invertibility of Toeplitz operators in terms of positive angle. If we further combine this lemma with the noncommutative Helson-Szeg\"o theorem obtained in the previous section, we end up with the promised structural characterisation of invertible Toeplitz operators with unitary symbols.

\begin{lemma}
Let $u\in \M$ be a unitary of the form described in the previous lemma. Then $T_u$ is invertible if and only if ${\A}^*$ and ${\A}_0$ are at positive angle with respect to the functional $\tau(w\cdot)$, where $w=g_0^*g_0=d^*(g_1^*g_1)^{-1}d$.
\end{lemma}

\begin{proof}
First suppose that $T_u$ is invertible. For any $a\in \A$ the element $g_0a$ will belong to $H^2(\M)$. So the invertibility of $T_u$ ensures that we can find a constant $K > 0$ so that 
 $$\|g_0a\|_2 \leq K\|T_u(g_0a)\|_2,\quad a\in \A.$$
Recall that by Lemma \ref{invtoep} $u$ is of the form $u=(g_1^*)^{-1}dg_0^{-1}$. Thus the former inequality translates to 
 $$\|g_0a\|_2 \leq K\|P_+((g_1^*)^{-1}da)\|_2,\quad a\in \A.$$ 
Now observe that for any $b\in \A_0$, the element $(g_1^*)^{-1}db^*$ will belong to $H^2(\M)^*\A_0^* \subset H^2_0(\M)^*$. Hence 
 $$P_+((g_1^*)^{-1}da) = P_+\big((g_1^*)^{-1}da + (g_1^*)^{-1}db^*\big).$$ 
If we now write $\|f\|_w$ for $\tau(wf^*f)^{1/2}$, then for any $a\in \A$ and $b\in \A_0$ we have that
\begin{eqnarray*}
\|a^*\|_w &=&\tau(a^*wa)^{1/2}=\|g_0a\|_2\\
&\leq& K\|P_+\big((g_1^*)^{-1}da + (g_1^*)^{-1}db^*\big)\|_2\\
&\leq& K\|(g_1^*)^{-1}d(a + b^*)\|_2\\
&=& K\tau((a^*+b)w(a+b^*))\\
&=& K\|a^*+b\|_w
\end{eqnarray*}
Thus ${\A}^*$ and $\A_0$ are at positive angle  with respect to the functional $\tau(w\cdot)$.

Conversely, suppose that $\A^*$ and $\A_0$ are at positive angle with respect to the functional $\tau(w\cdot)$. We first show that $T_u$ has dense range, and hence that it will be invertible whenever it is bounded below. Let $a_0\in H^2(\M)$ be orthogonal to $T_u(H^2(\M))$. We will show that $a_0$ must then be the zero vector. Given $a\in\A$, the orthogonality of $a_0$ to  $T_u(H^2(\M))$ together with the fact that $u=(g_1^*)^{-1}dg_0^{-1}$, ensures that
\begin{eqnarray*}
0 &=&\langle T_u(g_0a), a_0\rangle = \tau(a_0^*T_u(g_0a))\\
&=& \tau(a_0^*P_+((g_1^*)^{-1}da))\\
&=& \tau(a_0^*(g_1^*)^{-1}da).
\end{eqnarray*}
However, as was noted in the first part of the proof, for any $b\in\A_0$ we  have that 
 $$a_0^*(g_1^*)^{-1}db^*\in H^2_0(\M)^*, $$ 
which implies  that 
 $$\tau(a_0^*(g_1^*)^{-1}db^*)=\tau(\Phi(a_0^*(g_1^*)^{-1}db^*))=0.$$ 
Thus 
 $$\tau(a_0^*(g_1^*)^{-1}d(a+b^*))=0  \quad\mbox{for all}\quad  a\in\A, b\in \A_0.$$
Hence $d^*g_1^{-1}a_0=0$, so $a_0=0$.

It remains to show that $T_u$ is bounded below whenever ${\A}^*$ and $\A_0$ are at positive angle with
respect to the functional $\tau(w\cdot)$. Hence assume that there exists a constant $B>0$ so that 
 $$\|a^*\|_w\leq B\|a^*+b\|_w \quad\mbox{for all}\quad a\in \A, b\in \A_0.$$
Since by assumption we have that $d=\Phi(g_1^*)$, and since
both $g_1^*$ and $(g_1^*)^{-1}d$ belong to $H^2(\M)^*$, it follows that
 $$d=\Phi(d)=\Phi(g_1^*[(g_1^*)^{-1}d])=\Phi(g_1^*)\Phi((g_1^*)^{-1}d)=d\Phi((g_1^*)^{-1}d). $$
This yields that $\Phi((g_1^*)^{-1}d)=\I$. Now since $g_1^*$ is by assumption strongly outer, we have that $\Delta(g_1) =\Delta(\Phi(g_1))>0$ by \cite[Theorem 4.4]{BL3}. Consequently 
 $$\Delta(d) = \Delta(g_1^*)\Delta((g_1^*)^{-1}d)=\Delta(\Phi(g_1^*))\Delta((g_1^*)^{-1}d)=\Delta(d)\Delta((g_1^*)^{-1}d).$$
Thus since $\Delta(d) >0$ by the strong outerness of $d$, we must have that 
 $$\Delta((g_1^*)^{-1}d)=1=\Delta(\I)=\Delta(\Phi((g_1^*)^{-1}d)).$$ 
Hence by \cite[Theorem 4.4]{BL3} $(g_1^*)^{-1}d$ is a strongly outer element of $H^2(\M)^*$. But this ensures that $[(g_1^*)^{-1}d\A_0^*]=H^2_0(\M)^*$. Hence for any fixed $a\in \A$, we may select a sequence $\{b_n\}\subset A_0$ so
that 
 $$(g_1^*)^{-1}db_n^* \to (P_+-{\rm Id})[(g_1^*)^{-1}da]\in H^2_0(\M)^*\;\textrm{ in } L^2(\M).$$ 
Finally recall that by assumption $|g_0|=|(g_1^*)^{-1}d|$. So given any $a\in \A$, with
$\{b_n\}\subset \A_0$ the sequence as constructed above, we have that
\begin{eqnarray*}
\|g_0a\|_2 &=& \|a^*\|_w
\le B\|a^*+b_n\|_w\\
&=& B\|g_0(a+b_n^*)\|_2\\
&=& B\||g_0|(a+b_n^*)\|_2\\
&=& B\||(g_1^*)^{-1}d|(a+b_n^*)\|_2\\
&=& B\|(g_1^*)^{-1}d(a+b_n^*)\|_2.
\end{eqnarray*}
Letting $n\to \infty$ now yields $$\|g_0a\|_2\leq B\|P_+[(g_1^*)^{-1}da]\|=B\|T_u(g_0a)\|_2\quad\mbox{for any}\quad a\in \A.$$Finally note that by assumption $g_0$ is an outer element of $H^2(\M)$. With $g_0\A$ therefore being dense in $H^2(\M)$, the above inequality extends by continuity to the claim that 
 $$\|a\|_2\leq B\|T_u(a)\|_2\quad\mbox{for any}\quad a\in H^2(\M).$$
Thus $T_u$ is invertible.
\end{proof}

\begin{definition}
Given $f\in \M$ we define the {\em Hankel operator} with symbol $f$ by means of the prescription 
 $$\mathcal{H}_f:H^2(\M)\to H^2(\M)^*: x \mapsto P_-(fx),$$
where $P_-$ is the orthogonal projection from $L^2(\M)$ onto $H^2(\M)^*$.
\end{definition}

The following lemma is entirely elementary.

\begin{lemma}\label{Hankel}
Let $f\in \M$ be given. Then $$\|\mathcal{H}_f|_{H^2_0}\| = \sup\{|\tau(fF)|: F\in H^1_0(\M), \tau(|F|)\leq 1\}.$$
\end{lemma}

\begin{proof}
Since for every $x\in H^2(\M)$ we have that $({\rm Id}-P_-)(x) \in H^2_0(\M)$,  it is clear that such an $({\rm Id}-P_-)(x)$ will be orthogonal to any $y\in H^2(\M)^*$. Thus $\langle P_-(fa), b\rangle=\langle fa, b\rangle$ for any $a\in H^2_0(\M)$ and $b\in H^2(\M)^*$. Thus
\begin{eqnarray*}
\|\mathcal{H}_f|_{H^2_0}\| &=& \sup\{\|P_-(fa)\|: a\in H^2_0(\M), \|a\|_2\leq 1\}\\
&=& \sup\{|\langle P_-(fa), \,b\rangle|: a\in H^2_0(\M), b\in H^2(\M)^*, \|a\|_2\leq 1, \|b\|_2\leq 1\}\\
&=& \sup\{|\langle fa,\, b\rangle|: a\in H^2_0(\M), b\in H^2(\M)^*, \|a\|_2\leq 1, \|b\|_2\leq 1\}\\
&=& \sup\{|\tau(fab^*)|: a\in H^2_0(\M), b\in H^2(\M)^*, \|a\|_2\leq 1, \|b\|_2\leq 1\}\\
&=& \sup\{|\tau(fF)|: F\in H^1_0(\M), \tau(|F|)\leq 1\}.
\end{eqnarray*}
Here the last equality follows from the Noncommutative Riesz Factorisation theorem from \cite{S} and  \cite{MW1}.
 \end{proof}

We are now ready to present our final result. When taken alongside Theorem \ref{Toep1}, this result fully characterises invertible Toeplitz operators.

\begin{theorem}
Let $u\in \M$ be a unitary of the form described in Lemma \ref{invtoep}. Then the following are equivalent:
\begin{itemize}
\item $T_u$ is invertible;
\item there exists $k\in \A$ such that $\Re(u^*k)$ is strictly positive;
\item the Hankel operator $\mathcal{H}_u$ restricted to $H^2_0(\M)$ has norm less than $1$.
\end{itemize}
\end{theorem}

\begin{proof}
Our aim is to apply Theorem \ref{HS1}. In this regard we point out that although this theorem is formulated for norm one elements of $L^1(\M)^+$, that assumption is one of convenience and not necessity. Hence the value of $\|w\|_1$ is no essential obstruction to applying this theorem. Next observe that the fact that $w=g_0^*g_0$, not only ensures that $\Delta(w)=\Delta(g_0)^2>0$, but also that $w$ is injective.  Thus by Lemma \ref{Phisupp}, $s(\Phi(w)) = \I$. We showed in the proof of the preceding Lemma that $\Delta((g_1^*)^{-1}d)=1=\Delta(\Phi((g_1^*)^{-1}d))$. Applying this fact to $d^*g_1^{-1}$ enables us to conclude from \cite[Theorem 4.4]{BL3} that $d^*g_1^{-1}$ is a strongly outer element of $H^2(\M)$. On setting $h_R=d^*g_1^{-1}$ and $h_L=g_0$, it follows that $w$ is of the form 
 $$w = d^*g_1^{-1}(g_1^*)^{-1}d = d^*g_1^{-1}[(g_1^*)^{-1}dg_0^{-1}]g_0 = h_Ruh_L$$
with $h_R$ and $h_L$ strongly outer elements of $H^2(\M)$ for which we have that 
 $$|h_L|=|g_0|=w^{1/2}\;\textrm{ and }\; |h_R^*|=|(g_1^*)^{-1}d|=|w|^{1/2}.$$ 
With all the other conditions of this theorem being satisfied, we may now conclude from Theorem \ref{HS1} that $\A$ and $\A_0^*$ are at positive angle with respect to the functional $\tau(w\cdot)$ if and only if there exists a $k\in \A$ such that $\Re(u^*k)$ is strictly positive. From the proof of Theorem \ref{HS1} we also have that $\A$ and $\A_0^*$ are at positive angle if and only if $\sup\{|\tau(fF)|: F\in H^1_0(\M), \tau(|F|)\leq 1\}<1$. The result now follows from an application of the preceding two lemmata.
\end{proof}

\begin{remark}
We point out that for any unitary $u$ of the form described in Lemma \ref{invtoep}, the condition in the third bullet of the above theorem cannot be improved in the sense that for such a unitary, $\mathcal{H}_u$ must necessarily have norm 1.  Suppose that $u$ is of the form $u=(g_1^*)^{-1}dg_0^{-1}$  where $g_0,g_1$ are strongly outer elements of $H^2(\M)$ and $d$ a strongly outer element of $L^2(\D)$, related by the conditions that $dg^{-1}_0, d^*g^{-1}_1\in H^2(\M)$ and $g_0^*g_0=d^*(g_1^*g_1)^{-1}d$. Notice that $(g_1^*)^{-1}d = (d^*g_1^{-1})^* \in H_2(\M)^*$ must then be orthogonal to $({\rm Id}-P_-)(ug_0)$. Hence we get that
\begin{eqnarray*}
\langle \mathcal{H}_u(g_0), \,(g_1^*)^{-1}d\rangle &=& \langle ug_0, \,(g_1^*)^{-1}d\rangle\\
&=& \langle (g_1^*)^{-1}d, \,(g_1^*)^{-1}d\rangle\\
&=& \tau(d^*g_1^{-1}(g_1^*)^{-1}d)\\
&=& \tau(g_0^*g_0)^{1/2}.\tau(d^*g_1^{-1}(g_1^*)^{-1}d)^{1/2}\\
&=& \|g_0\|_2.\|(g_1^*)^{-1}d\|_2.
\end{eqnarray*}
This can clearly only be the case if $\|\mathcal{H}_u\|\geq 1$. Since we also have that $\|\mathcal{H}_u\| \leq \|u\|_\infty =1$, the claim follows.
\end{remark}

\noindent{\bf Acknowledgments.} The contributions of the first named author is based upon research supported by the National Research Foundation. Any opinion, findings and conclusions or recommendations expressed in this material, are those of the authors, and therefore the NRF do not accept any liability in regard thereto. The second named author is partially supported by ANR-2011-BS01-008-01 and NSFC grant No.  11271292.


\bigskip

\begin{thebibliography}{999}

\bibitem{ar}
W.~B. Arveson.
 \newblock Analyticity in operator algebras.
 \newblock \textit{Amer. J. Math.},  \textbf{89} (1967), 578-642.

\bibitem{BruDom}
 R. Bruzual and M. Dom\'inguez.
 \newblock Operator-valued extension of the theorem of Helson and Szeg\"o.
  \newblock \textit{Operator Theory: Advances and Applications.}  \textbf{149} (2004), 139-152.

\bibitem{bx}
 T. N. Bekjan and Q. Xu.
 \newblock Riesz and Szeg\"o type factorizations for noncommutative Hardy spaces.
 \newblock  \textit{J. Operator Theory}.  \textbf{62} (2009),  215-231.

\bibitem{bekU}
 M. Bekker and A.P. Ugol'nikov.
 \newblock The Helson-Szeg\"o theorem for operator-valued weight.
 \newblock  \textit{Methods of Funct. Anal. and Topology}.  \textbf{10} (2004),  11-16.


\bibitem{BL2}
 D. P. Blecher and L. E. Labuschagne.
 \newblock  Characterizations of noncommutative $H^\infty$.
 \newblock \textit{Integr. Equ. Oper. Theory} \textbf{56} (2006), 301-321.

\bibitem{BL3}
 D. P. Blecher and L. E. Labuschagne.
  \newblock Applications of the Fuglede-Kadison determinant: Szeg\"o's theorem and outers for noncommutative $H^p$.
  \newblock \textit{Trans. Amer. Math. Soc.} \textbf{360} (2008), 6131-6147.

\bibitem{BLsur}
 D. P. Blecher and L. E. Labuschagne.
 \newblock Von Neumann algebraic $H^p$ theory, Proceedings of the 5th conference on function spaces.
 \newblock  \textit{Contemporary Math.} \textbf{435} (2007) 89-114.

\bibitem{Br}
 L. G. Brown.
 \newblock  Lidski\u{i}'s theorem in the type II case.
 \newblock  \textit{Geometric methods in operator algebras} (Kyoto, 1983), 1-35. Pitman Res. Notes Math. Ser., 123, Longman Sci. Tech., Harlow, 1986.

\bibitem{Dev}
A. Devinatz.
\newblock Toeplitz operators on $H^2$ spaces.
\newblock  \textit{Trans. Amer. Math. Soc.} \textbf{112} (1964), 304-317.

\bibitem{Dom}
 M. A. Mom\'inguez.
 \newblock A matricial extension of the Helson-Sarason theorem and a characterization of some multivariate linearly completely regular processes.
 \newblock  \textit{J. Multivariate Anal.} \textbf{31} (1989), 289-310.

\bibitem{fug-kad}
 B. Fuglede  and R.V. Kadison.
 \newblock Determinant theory in finite factors.
 \newblock \textit{Ann. Math.} \textbf{55} (1952), 520-530.

\bibitem{gar}
 J.B. Garnett.
 \newblock \textit{Bounded analytic functions.}
 \newblock Academic Press, 1981.

\bibitem{HS}
 U. Haagerup and H. Schultz.
  \newblock  Brown measures of unbounded operators affiliated with a finite von Neumann algebra.
  \newblock \textit{Math. Scand.} \textbf{100} (2007), 209-263.

\bibitem{HSz}
 H. Helson and G. Szeg\"o.
 \newblock A problem in prediction theory
 \newblock \textit{Ann. Mat. Pure Appli.} \textbf{51} (1960), 107-138.

\bibitem{HMW}
 R. Hunt, B. Muckenhoupt and R. Wheeden.
 \newblock Weighted norm inequalities for conjugate function and Hilbert transform.
 \newblock \textit{Trans. Amer. Math. Soc.} \textbf{176} (1973), 227-251.

\bibitem{L1}
 L. E. Labuschagne.
 \newblock  A noncommutative Szeg\"o theorem for subdiagonal subalgebras of von Neumann algebras.
 \newblock \textit{Proc. Amer. Math. Soc.}  \textbf{133} (2005), 3643-3646.

\bibitem{MW1}
 M. Marsalli and G. West.
 \newblock Non-commutative $H^p$ spaces.
 \newblock \textit{J Operator Theory} \textbf{40} (1998), 339-355.

\bibitem{MW}
 M. Marsalli and G. West.
 \newblock Toeplitz operators with noncommuting symbols.
 \newblock  \textit{Integr. Equ. Oper. Theory} \textbf{32} (1998), 65-74.

\bibitem{px}
G.~Pisier and Q.~Xu.
 \newblock Non-commutative {$L\sp p$}-spaces.
 \newblock In \textit{Handbook of the geometry of Banach spaces, Vol. 2}, pages 1459-1517. North-Holland, Amsterdam, 2003.

\bibitem{Pour}
 M. Pouraimadi.
 \newblock A Matricial extension of the Helson-Szegd theorem and its application in multivariate prediction.
 \newblock  \textit{J. Multivariate Anal.} \textbf{16} (1985), 265-275.

\bibitem{P1}
 H. Pousson.
 \newblock  Systems of Toeplitz Operators on $H^2$.
 \newblock  \textit{Proc. Amer. Math. Soc.}  \textbf{19} (1968), 603-608.


\bibitem{P2}
 H. Pousson.
 \newblock  Systems of Toeplitz Operators on $H^2$: II.
 \newblock  \textit{Trans. Amer. Math. Soc.}  \textbf{133} (1968), 527-536.

\bibitem{Pr}
 B. Prunaru. 
 \newblock  Toeplitz and Hankel operators associated with subdiagonal algebras.
 \newblock  \textit{Proc. Amer. Math. Soc.} \textbf{139} (2010), 1387-1396.
 
\bibitem{ran}
 N. Randrianantoanina.
 \newblock  Hilbert transform associated with finite maximal subdiagonal algebras.
 \newblock  \textit{J. Austral. Math. Soc. Ser. A.} \textbf{65} (1998), 388-404.

\bibitem{S}
 K.-S. Saito.
 \newblock  A note on invariant subspaces for finite maximal subdiagonal algebras.
 \newblock  \textit{Proc. Amer. Math. Soc.} \textbf{77} (1979), 348-352.

\bibitem{Tak1}
 M. Takesaki.
 \newblock  \textit{Theory of Operator Algebras: Vol 1.}
 \newblock Springer, New York, 1979.

\bibitem{TV}
S. Treil and A. Volberg.
\newblock  Wavelets and the Angle between Past and Future.
\newblock  \textit{J. Funct. Anal.} \textbf{143} (1997), 269-308.

\bibitem{Ueda}
 Y. Ueda.
  \newblock On peak phenomena for non-commutative $H^\infty$.
  \newblock  \textit{ Math. Ann.} \textbf{343} (2009), 421-429


\end{thebibliography}
\end{document}